\documentclass{article}

\usepackage{amssymb}
\usepackage{amsthm}
\usepackage{amsmath}
\usepackage{array,hhline}
\newtheorem{lemma}{Lemma}[section]
\newtheorem{theorem}[lemma]{Theorem}
\newtheorem{proposition}[lemma]{Proposition}

\newtheorem{corollary}[lemma]{Corollary}
\theoremstyle{definition}
\newtheorem{definition}[lemma]{Definition}

\numberwithin{equation}{section}
\numberwithin{figure}{section}

\newcommand{\Kset}{\mathcal{K}}

\newcommand{\Pset}{\mathcal{P}}

\newcommand{\Xset}{\mathcal{X}}
\newcommand{\Yset}{\mathcal{Y}}
\newcommand{\Zset}{\mathcal{Z}}

\begin{document}

\author{\centerline {\large H. Hakopian and S. Toroyan} \\ {Yerevan State University, Yerevan, Armenia}}

\title{{\huge On the uniqueness of algebraic curves passing through $n$-independent nodes}}

\maketitle

%%%%%%%%%%%%%%%%%%%%%%%%%%%%%%%%%%%%%%%%%%%%%%%%%%%%%%%%%%%%%%%%%%%%%%%%
\begin{abstract}
A set of nodes is called $n$-independent if each its node has a fundamental polynomial of degree $n.$
 We proved in a previous paper [H. Hakopian and S.
Toroyan, On the minimal number of nodes determining uniquelly
algebraic curves, accepted in Proceedings of YSU] that the
minimal number of $n$-independent nodes determining uniquely the curve
of degree $k\le n$ equals to $\Kset:=(1/2)(k-1)(2n+4-k)+2.$ Or,
more precisely, for any $n$-independent set of cardinality $\Kset$ there is at most one curve of degree $k\le n$
passing through its nodes,
while there are $n$-independent
node sets of cardinality $\Kset-1$ through which pass at least
two such curves. In this paper we bring a simple characterization of the latter sets.
Namely, we prove that if two curves of degree $k\le n$ pass through the nodes of an $n$-independent
node set $\Xset$ of cardinality $\Kset-1$ then all the nodes of $\Xset$ but one belong to a (maximal) curve of degree $k-1.$
\end{abstract}



\section{Introduction}

Denote the space of all bivariate polynomials of total degree $\leq n$ by $\Pi_n$:

\begin{equation*}
  \Pi_n= \left\{ \sum_{i+j\leq n} a_{ij} x^i y^j \right\}.
\end{equation*}
We have that
\begin{equation*}
  N := N_n := \dim \Pi_n = (1/2)(n+1)(n+2).
\end{equation*}
Consider a set of $s$ distinct nodes
\begin{equation*}
  \Xset_s = \{(x_1,y_1), (x_2,y_2), \ldots, (x_s,y_s) \}.
\end{equation*}
The problem of finding a polynomial $p \in \Pi_n$ which satisfies the conditions
\begin{equation}\label{eq:intpr}
  p(x_i,y_i) = c_i, \quad i=1, \ldots, s,
\end{equation}
is called interpolation problem.

A polynomial $p \in \Pi_n$ is called an $n$-fundamental polynomial for a node $A=(x_k,y_k)\in\Xset_s$ if
\begin{equation*}
  p(x_i,y_i)= \delta_{ik}, \ i=1,\ldots, s,
\end{equation*}
where $\delta$ is the Kronecker symbol. We denote this fundamental polynomial
by $p_{k}^\star=p_{A}^\star=p_{A,\Xset_s}^\star$. Sometimes we call fundamental also a polynomial that vanishes at all nodes
of $\Xset_s$ but one, since it is a nonzero constant times a fundamental polynomial.

Next, let us consider an important concept of $n$-independence (see \cite{E}, \cite{HM}).

\begin{definition}
A set of nodes $\Xset$ is called \emph{$n$-independent} if all its
nodes have $n$-fundamental polynomials. Otherwise, if a node has no
$n$-fundamental polynomial, $\Xset$ is called \emph{$n$-dependent}.
\end{definition}

Fundamental polynomials are linearly independent. Therefore a necessary condition of $n$-independence of $\Xset_s$
is $s \leq N$.

Suppose a node set $\Xset_s$ is n-independent. Then by the Lagrange
formula we obtain a polynomial $p\in\Pi_n$ satisfying the
interpolation conditions $\eqref{eq:intpr}$:
\begin{equation*}
  p = \sum_{i=1}^sc_ip_{i}^\star.
\end{equation*}
In view of this, we get readily that the node set $\Xset_s$ is
$n$-independent if and only if the interpolating problem
$\eqref{eq:intpr}$ is \emph{solvable}, meaning that for any data
$(c_1,\ldots, c_s)$ there is a polynomial $p\in\Pi_n$ (not
necessarily unique) satisfying the interpolation conditions
\eqref{eq:intpr}.

\begin{definition} \label{poised}
The interpolation problem with a set of nodes $\Xset_s$ and $\Pi_n$
is called \emph{$n$-poised} if for any data $(c_1,\ldots, c_s)$
there is a \emph{unique} polynomial $p\in\Pi_n$ satisfying the
interpolation conditions \eqref{eq:intpr}.
\end{definition}
The conditions \eqref{eq:intpr} give a system of $s$ linear
equations with $N$ unknowns (the coefficients of the polynomial
$p$). The poisedness means that this system has a unique solution
for arbitrary right side values. Therefore a necessary condition of
poisedness is $ s = N.$ If this condition holds then we obtain from
the linear system

\begin{proposition}\label{prp:int0}
A set of nodes $\Xset_N$ is $n$-poised if and only if
\begin{equation*}
 p\in\Pi_n\ \ \text{and}\ \ p\big\vert_{\Xset_N} = 0
 \quad\implies\quad p = 0.
\end{equation*}
\end{proposition}

Thus, geometrically, the node set $\Xset_N$ is $n$-poised if and only if there is no curve of degree $n$ passing through all its nodes.

It is worth mentioning
\begin{proposition}\label{Ecurve}
For any set $\Xset_{N-1},$ i.e., set of cardinality $N-1,$ there is a curve of degree $n$ passing through all its nodes.
 \end{proposition}
Indeed, the existence of the curve reduces to a system of $N-1$ linear homogeneous equations with $N$ unknowns -- the coefficients of the polynomial of degree $n.$

It follows from Proposition \ref{prp:int0} also that a node set of cardinality $N$ is $n$-poised if and only if it is $n$-independent.

Suppose we have an $m$-poised set ${\mathcal X}_N.$
From what was said above we can conclude easily that through any $N-1$ nodes of $\mathcal X$ there pass a unique curve of degree $n.$
Namely the curve given by the fundamental polynomial of the missing node. While through any $N-2$ nodes of $\mathcal X$ there pass more than one curve of degree $n,$ for example the curves given by the fundamental polynomials of two missing nodes. Thus we have that the
minimal number of $n$-independent nodes determining uniquely the curve
of degree $n$ equals to $N-1.$

In \cite{HT} we considered this problem in the case of arbitrary degree $k, k\le n.$ We proved that the
minimal number of $n$-independent nodes determining uniquely the curve
of degree $k\le n$ equals to $\Kset:=(1/2)(k-1)(2n+4-k)+2.$ Or,
more precisely, for any $n$-independent set of cardinality $\Kset$ there is at most one curve of degree $k\le n$
passing through its nodes,
while there are $n$-independent
node sets of cardinality $\Kset-1$ through which pass at least
two such curves. Let us mention that the above described problem in the case $k=n-1$ was solved in \cite{BHT}.

In this paper we bring a simple characterization of the sets of cardinality  $\Kset-1$ through which pass at least
two curves of degree $k.$
Namely, we prove that in this case all the nodes of $\Xset$ but one belong to a curve of degree $k-1.$ Moreover, this
latter curve is a maximal curve meaning that it passes through maximal possible number of $n$-independent nodes 
(see Section 3).

At the end let us bring a well-known Berzolari-Radon construction of $n$-poised set (see \cite{B}, \cite{R}).
\begin{definition}\label{BR}
A set of $N=1+\cdots+(n+1)$ nodes is called Berzolari-Radon set for
degree $n$, or briefly $BR_n$ set, if there exist lines
$l_{1},l_{2},\ldots,l_{n+1}$, such that the sets $l_{1},\
l_2\setminus l_{1},\ l_{3}\setminus (l_{1}\cup l_{2}),\ldots,\
l_{n+1}\setminus(l_{1}\cup\cdots\cup l_{n})$ contain exactly
$(n+1),n,n-1,\ldots,1$ nodes, respectively.
\end{definition}

\section{ Some properties of $n$-independent nodes}

Let us start with the following simple (see \cite{HMush}, Lemma 2.3)
\begin{lemma}\label{XA}
Suppose that a node set ${\mathcal X}$ is $n$-independent and a node $A\notin \mathcal X$
has $n$-fundamental polynomial with
respect to the set ${\mathcal X}\cup \{A\}.$ Then the latter
node set is $n$-independent, too.
\end{lemma}
Indeed, one can get readily the fundamental polynomial of any node $B\in \Xset$ with respect to the set $\Yset:={\mathcal X}\cup \{A\}$ by using
the given fundamental polynomial $p^\star_A$ and the fundamental polynomial of $B$ with respect to the set ${\mathcal X}.$

Evidently, any subset of $n$-poised set is $n$-independent.
According to the next lemma any $n$-independent set is a subset of
some $n$-poised set (see, e.g. \cite{HJZ}, Lemma 2.1):
\begin{lemma}\label{ext}
Any $n$-independent set $\Xset$ with $\#\Xset<N$ can be enlarged to an $n$-poised set.
\end{lemma}
\begin{proof}
It suffices to show that there is a node $A$ such that the set ${\mathcal X}\cup \{A\}$
is $n$-independent. By Proposition \ref{Ecurve} there is a nonzero polynomial $q\in\Pi_n$ such that $q\big\vert_{\Xset} = 0.$
Now, in view of Lemma \ref{XA}, we may choose a desirable node $A$ by requiring only that $q(A)\ne 0.$ Indeed, then $q$ is a fundamental polynomial
of $A$ with respect to the set ${\mathcal X}\cup \{A\}.$
 \end{proof}

Denote the linear space of polynomials of total degree at most $n$
vanishing on ${\mathcal X}$ by
\begin{equation*}{{\mathcal P}}_{n,{\mathcal X}}=\left\{p\in \Pi_n
: p\big\vert_{\mathcal X}=0\right\}.
\end{equation*}
The following is well-known (see e.g. \cite{HM})
\vspace{5mm}

\begin{proposition} \label{PnX} For any node set ${\mathcal X}$ we have that
\begin{equation*}\label{eq:theta1} \dim {{\mathcal P}}_{n,{\mathcal X}} \ge N - \#{\mathcal X}.\end{equation*}
Moreover, equality takes place here if and only if the set ${\mathcal X}$ is $n$-independent.
\end{proposition}

From here one gets readily (see \cite{HMush}, Corollary 2.4):

\begin{corollary} \label{cor:ind4} Let ${\mathcal Y}$ be a maximal $n$-independent subset of ${\mathcal X},$ i.e.,  ${\mathcal Y}\subset{\mathcal X}$ is $n$-independent and
${\mathcal Y}\cup \{A\}$ is $n$-dependent for any $A\in {\mathcal
X}\setminus {\mathcal Y}.$ Then we have that
\begin{equation}\label{eq:theta2}{{\mathcal P}}_{n,{\mathcal Y}}= {{\mathcal P}}_{n,{\mathcal X}}.
\end{equation}
\end{corollary}
\begin{proof} We have that ${{\mathcal P}}_{n,{\mathcal
X}}\subset {{\mathcal P}}_{n,{\mathcal Y}},$ since ${\mathcal
Y}\subset {\mathcal X}.$ Now, suppose that $p\in \Pi_n,\
p\big\vert_{\mathcal Y}=0$ and $A$ is any node of ${\mathcal X}.$
Then ${\mathcal Y}\cup\{A\}$ is dependent and therefore, in view of
Lemma \ref{XA}, $p\big\vert_{A}=0.$
\end{proof}

\noindent From \eqref{eq:theta2} and Proposition \ref{PnX}
(part "moreover") we have that
\begin{equation} \label{eq:theta3} \dim {{\mathcal P}}_{n,{\mathcal X}} = N - \#{\mathcal Y},\end{equation}
where ${\mathcal Y}$ is any maximal $n$-independent subset of
${\mathcal X}.$ Thus, all the maximal $n$-independent subsets of
${\mathcal X}$ have the same cardinality, which is denoted by
${\mathcal H}_n ({\mathcal X})\ -$ \emph{the Hilbert $n$-function}
of ${\mathcal X}.$ Hence, according to \eqref{eq:theta3}, we have that
\begin{equation*} \dim {{\mathcal P}}_{n,{\mathcal X}} = N - {\mathcal H}_n ({\mathcal
X}).\end{equation*}

\section{Maximal curves}

An algebraic curve in the plane is the zero set of some bivariate
polynomial of degree at least $1.$ We use the same letter, say $p,$
to denote the polynomial $p\in\Pi_k\setminus \Pi_{k-1}$ and the
corresponding curve  $p$ of degree $k$ defined by equation
$p(x,y)=0.$

According to the following well-known statement there are no more
than $n+1\ n$-independent points in any line:
\begin{proposition}\label{maxline}
Assume that $l$ is a line and ${\mathcal X}_{n+1}$ is any subset of
$l$ containing $n+1$ points. Then we have that
$$p\in {\Pi_{n}}\quad \text{and} \quad p|_{{\mathcal X}_{n+1}}= 0 \; \Longrightarrow \quad p = lr,$$
\end{proposition}
\noindent where $r \in \Pi_{n-1}.$

\noindent Denote
\begin{equation*} d:=d(n, k) := N_n - N_{n-k} = k(2n+3-k)/2.
\end{equation*}
The following is a generalization of Proposition \ref{maxline}.
\begin{proposition}[\cite{Raf}, Prop. 3.1]\label{maxcurve}
Let $q$ be an algebraic curve of degree $k \le n$ without multiple
components. Then the following hold.\\
i) Any subset of $q$ containing more than $d(n,k)$ nodes is
$n$-dependent.\\
ii) Any subset ${\mathcal X}_d$ of $q$ containing exactly $d(n,k)$
nodes is $n$-independent if and only if the following condition
holds:
        $$p\in {\Pi_{n}}\quad \text{and} \quad p|_{{\mathcal X}_d} = 0 \Longrightarrow  p = qr, $$
\end{proposition}
\noindent where $r \in \Pi_{n-k}.$

Suppose that $\mathcal X$ is an $n$-poised set of nodes and $q$ is
an algebraic curve of degree $k \le n$. Then of course any subset of
$\mathcal X$ is $n$-independent too. Therefore, according to
Proposition~\ref{maxcurve} i), at most $d(n,k)$ nodes of $\mathcal
X$ can lie in the curve $q$. \noindent Let us mention that a special
case of this when $q$ is a set of $k$ lines is proved in
\cite{CG2001b}.

This motivates the following definition (see \cite{Raf}, Def. 3.1).
\begin{definition}\label{def:maximal}
Given an $n$-independent set of nodes $\mathcal X_s,$ with $s\ge
d(n,k).$ A curve of degree $k \le n$ passing through $d(n,k)$ points
of $\mathcal X_s,$ is called maximal.
\end{definition}

\noindent Note that maximal line, as a line passing through $n+1$
nodes, is defined in \cite{CG2001}.

We say that a node $A\in \Xset$ uses a polynomial $q\in \Pi_k$ if the latter divides the fundamental polynomial
$p=p^\star_{A},$ i.e., $p = qr,$
for some $r \in \Pi_{n-k}.$

Next, we bring a characterization of maximal curves:

\begin{proposition} [\cite{Raf}, Prop. 3.3] \label{maxcor}  Let a node set $\Xset$ be n-poised. Then
a polynomial $\mu$ of degree $k,\ k\le n,$ is a maximal curve if and only if it is
used by any node in  $\Xset\setminus \mu.$
\end{proposition}
Note that one side of this statement follows from Proposition \ref{maxcurve} (ii). In
the case of lines this was proved in \cite{CG2001}. For other properties of maximal
curves we refer reader to \cite{Raf}.

\begin{proposition} \label{extcurve} Assume that $\sigma$ is an algebraic curve  of degree $k,$ without multiple components, and
$\Xset_s\subset \sigma$ is any $n$-independent node set of cardinality
$s,\ s<d(n,k).$ Then the set $\Xset_s$ can be extended to a maximal
$n$-independent set $\Xset_{d}\subset \sigma$ of cardinality $d,$ where $d=d(n,k)$.
\end{proposition}

\begin{proof}
It suffices to show that there is a point $A\in\sigma\setminus \Xset_s$ such
that the set $\Xset_{s+1}:=\Xset_s\cup \{A\}$ is $n$-independent.
Assume to the contrary that there is no such point, i.e., the set
$\Xset_{s+1}:=\Xset_s\cup \{A\}$ is $n$-dependent for any $A\in
\sigma.$  Then, in view of Lemma \ref{XA}, $A$ has no fundamental
polynomial with respect to the set $\Xset_{s+1}.$ In other words we
have
\begin{equation*}\label{Xy}
 p\in\Pi_n\ \ \text{and}\ \ p\big\vert_{{\mathcal X}_s} = 0
  \quad \implies \quad p(A)=0 \ \ \text{for any}\ \  A\in
\sigma.
\end{equation*}
From here we obtain that
$${\mathcal P}_{n,\Xset_s}\subset {\mathcal P}_{n,\sigma}:=\left\{{q\sigma : q\in \Pi_{n-k}}\right\}.$$
Now, in view of Proposition \ref{PnX}, we get from here
$$N-s=\dim{\mathcal P}_{n,\sigma}\le \dim{\mathcal P}_{n,\Xset_s}= N_{n-k}.$$
Therefore $s\ge d(n,k),$ which contradicts the hypothesis of
Proposition.\end{proof}

The following lemma follows readily from the fact that the
Vandermonde determinant, i.e., the main determinant of the linear
system described after Definition \ref{poised}, is a continuous
function of the nodes of $\Xset_N$ (see e.g., \cite{H82}, Remark 1.14).
\begin{lemma} \label{eps} Suppose $\Xset_N=\{(x_i,y_i)\}_{i=1}^N$ is $n$-poised. Then there is a positive number $\epsilon$ such that
any set $\Xset_N'=\{(x_i',y_i')\}_{i=1}^N,$ for which distance between $(x_i',y_i')$ and $(x_i,y_i)$ is less than $\epsilon,$ is $n$-poised too.
\end{lemma}

Finally, let us bring a lemma that follows from a simple Linear
Algebra argument (see e.g., \cite{HJZ2}, Lemma 2.10).
\begin{lemma} \label{2cor} Suppose that two different curves of degree $k$ pass through all the nodes of $\Xset.$
Then for any node $A\notin \Xset$ there is a curve of degree $k$ passing through all the nodes of $\Xset$ and $A.$
\end{lemma}

\section{Main result}

In a previous paper \cite{HT} we determined the minimal number of
$n$-independent nodes that uniquely determine the curve of degree
$k, \ k\le n,$ passing through them:

\begin{theorem}\label{ht}
Assume that $\Xset$ is any set of $(d(n, k-1)+2) \ \ n$-independent
nodes lying in a curve of degree $k$ with $k\le n.$ Then the curve
is determined uniquely. Moreover, there is a set $\tilde\Xset$ of $(d(n,
k-1)+1) \ \ n$-independent nodes such that more than one curves of
degree $k$ pass through all its nodes.
\end{theorem}
Let us mention that this result, in the case $k=n-1,$ was established in \cite{BHT}.

In this section we give a characterization of the case when more
than one curve of degree $k, \ k\le n,$ passes through the nodes of
an $n$-independent set $\Xset$ of cardinality $d(n, k-1)+1.$

As we will see later this result is a generalization of Theorem \ref{ht}

\begin{theorem}\label{mainth}
Given a set of $n$-independent nodes $\Xset$ with $\#\Xset = d(n, k-1)+1.$
Then there are at least $2$ curves of degree
$k$ passing through all nodes of $\Xset$ if and only if  there exists a
maximal curve $\mu$ of degree $k-1$ passing through $d(n, k-1)$ nodes of
$\Xset$ and the remaining node
of $\Xset$ is outside of  $\mu$.
\end{theorem}

\begin{proof}
Let us start with the inverse implication. Assume that $d(n, k-1)$
nodes of $\Xset$ are located on a curve $\mu$ of degree $k-1$.
Therefore, as it is mentioned in the formulation of Theorem, the
curve $\mu$ is maximal and the remaining node of $\Xset,$ which we
denote by $A,$ is outside of it: $A\notin \mu.$ Now, according to
Proposition \ref{maxcor}, we have
$$\Pset_{k,\Xset}= \left\{\alpha\mu | \alpha\in\Pi_1,\ \alpha(A)=0\right\}.$$
Therefore we get readily
$$\dim\Pset_{k,\Xset}=\dim\left\{\alpha | \alpha\in\Pi_1,\ \alpha(A)=0\right\}=2.$$

Now let us prove the direct implication. Assume that there are two
curves of degree $k:\ \sigma_1$ and $\sigma_2$ that pass through
all the nodes of the $n$-independent set $\Xset$, $\#\Xset = d(n,
k-1)+1$. Next, choose a node $B\notin\Xset$ such that the
following three conditions are satisfied:

(i) $B$ does not belong to any line passing through two nodes of
$\Xset,$

(ii) $B$ does not belong to the curves $\sigma_1$ and $\sigma_2,$

(iii) The set $\Xset\cup \{B\}$ is $n$-independent.

Let us verify that one can find a such node. Indeed, in view of Lemma \ref{ext}, we can start by choosing a node $B'$ satisfying the condition
(iii). Then notice that, according to Lemma \ref{eps}, for some positive $\epsilon$ all the nodes
in $\epsilon$ neighborhood of $B'$ satisfy the condition (iii).
Finally, from this neighborhood we can choose a node $B$
satisfying the condition (i) and (ii), too.

In view of Proposition \ref{2cor} there is a curve of degree $k$ passing through all
the nodes of $\Yset:=\Xset\cup \{B\}.$ Denote a such curve by
$\sigma.$ In view of (ii) $\sigma$ is different from $\sigma_1$
and $\sigma_2.$

Next, by using Proposition \ref{extcurve}, let us extend the set $\Yset$ till a maximal
$n$-independent set $\Zset\subset\sigma.$ Notice that, since
$\#\Zset = d(n,k)$, we need to add $d(n, k)-(d(n, k-1)+2) = n-k$
nodes, denoted by $C_1,\ldots,C_{n-k}$:

$$\Zset:=\Xset\cup \{B\} \cup \left\{C_i\right\}_{i=1}^{n-k}.$$

 Thus the curve $\sigma$ becomes maximal
with respect to the set $\Zset.$

Then let us consider $n-k-1$ lines $\ell_1, \ell_2, \ldots ,
\ell_{n-k-1}$ passing through the nodes $C_1, C_2,\ldots,
C_{n-k-1},$ respectively. We require that each line passes through
only one of the mentioned nodes and therefore the lines are
distinct. We require also that none of these lines is a component
(factor) of $\sigma.$ Finally let us denote by $\tilde{\ell}$
the line passing through $B$ and $C_{n-k}.$

Now notice that the following polynomial of degree $n$ vanishes at
all points of $\Zset$
\begin{equation}\label{C_2}
\sigma_1 \, \tilde{\ell} \, \ell_1 \, \ell_2 \dots \,
\ell_{n-k-1}.
\end{equation}
Consequently, in view of Proposition \ref{maxcor}, $\sigma$ divides this polynomial:
\begin{equation}\label{sigmalq}
\sigma_1 \, \tilde{\ell} \, \ell_1 \, \ell_2 \dots \, \ell_{n-k-1}
= \sigma \, q , \quad q \in\Pi_{n-k}.
\end{equation}
The distinct lines  $\ell_1, \ell_2, \dots , \ell_{n-k-1}$ do not
divide the polynomial $\sigma \in\Pi_{k}$, therefore all they have
to divide $q \in\Pi_{n-k}.$ Thus $q=\ell_1\dots\ell_{n-k-1}\ell',$
where $\ell'\in\Pi_1.$ Therefore, we get from \eqref{sigmalq}:
\begin{equation}\label{C_max}
\sigma_1 \, \tilde{\ell} = \sigma \, \ell'.
\end{equation}
If the lines $\tilde{\ell}, \ell'$ coincide then the curves
$\sigma_1, \sigma$ coincide, which is impossible. Therefore the
line $\tilde{\ell}$ has to divide $\sigma \in \Pi_k$:
\begin{equation*}
\sigma = \tilde{\ell} \, r, \quad r \in \Pi_{k-1}.
\end{equation*}
Let us study this relation closer. We are going to derive from here that
the curve $r$ passes through all the nodes of the set $\Xset$
but one. Indeed, $\sigma$ passes through all
the nodes of $\Xset.$ Therefore these nodes are either in the curve $r$ or in the line $\tilde
\ell.$ But this line passes through $B$, and
according to (i), it passes  through at most one node of $\Xset.$
Thus $r$ passes through at least $d(n,k-1)$ nodes of $\Xset$ and therefore it is a maximal curve of degree $k-1.$
On the other hand, according to Proposition \ref{maxcurve}, the curve $r$ of degree $k-1$ can pass through at most $d(n,k-1)$ independent nodes. Thus, we conclude that $r$ passes through exactly $d(n,k-1)$ nodes of $\Xset$.
\end{proof}

\section{Two corollaries}

As it was mentioned earlier, our main result -- Theorem \ref{mainth} yields the uniqueness result: Theorem \ref{ht}, which states that the minimal number of $n$-independent points
 determining uniquely a curve of degree $k, k \leq n-1$ equals to $d(n, k-1) + 2$ (see \cite{HT}, Theorem 2.1):

\begin{corollary}\label{maincor}
Given a set of $n$-independent nodes $\Xset$, $\#\Xset = d(n, k-1)+2.$
Then there can be at most one curve of degree
$k$ which passes through all its nodes.
\end{corollary}

\begin{proof} Choose a node $A\in \Xset$ and consider the set $\Yset:=\Xset\setminus\{A\}.$
If there is at most one curve of degree
$k$ which passes through all nodes of $\Yset$ then we are done. Next suppose, that there are at least two curves of degree
$k$ which pass through all nodes of the set $\Yset.$ Then, according to Theorem \ref{mainth}, there is a node $B\in \Yset$
and a maximal curve  $\mu_{k-1}$ of degree $k-1$ which passes through all the nodes of $\Yset\setminus\{B\}.$
Moreover, all the nodes of $\Xset$ but $A$ and $B$ are located in the curve $\mu_{k-1}.$
Now, in view of Proposition \ref{maxcor}, any curve of degree $k$ passing through all the nodes of $\Xset$ has the following form
$$p=\ell\mu_{k-1},$$
where $\ell\in\Pi_1.$
Finally notice that the line $\ell$ passes through $A$ and $B$ and therefore is determined in a unique way.
Hence $p$ is determined uniquelly.
\end{proof}

\begin{corollary}
Let $\Xset$ be an $n$-poised set of nodes and  $\ell$ be a used line which passes through exactly $3$ nodes. Then
it is used either by exactly one or by exactly three nodes from $\Xset.$  Moreover, if it is
used by exactly three nodes, then they are noncollinear.
\end{corollary}

\begin{proof}
Assume that $\ell\cap\Xset=\{A,B,C\}.$ Assume also that there are two nodes $P,Q\in \Xset$ using the line
$\ell:$
\begin{equation*}
p_{P}^{\star}=\ell \, q_1, \quad p_{Q}^{\star}={\ell} \, q_2,
\end{equation*}
where $q_1,q_2 \in \Pi_{n-1}.$

Both the polynomials $q_1, q_2$ vanish at $N-5$ nodes of the set $\Yset:=\Xset\setminus\{A,B,C,P,Q\}.$ Hence these $N-5=d(n, n-2)+1$ nodes do not uniquelly determine curve
of degree $n-1$ passing through them.
By the Proposition~\ref{mainth} there exists a maximal curve
$\mu_{n-2}$ of degree $n-2$ passing through $N-6$ nodes of $\Yset$ and the remaining node denoted by $R$ is outside.
Now, according to Proposition \ref{maxcor}, $\mu_{n-2}$ divides the fundamental polynomial of the node $R:$
\begin{equation}
p_{R}^{\star} = \mu_{n-2} \,q,
\end{equation}
where $q \in \Pi_2.$
This quadratic polynomial $q$ has to vanish at the three nodes $A,B,C\in\ell.$
Hence $q=\ell\ell'$ with $\ell'\in \Pi_1.$
Therefore, in view of Proposition \ref{maxline}, with $n=2,$ the node $R$ uses the line ${\ell}:$
\begin{equation}
p_{R}^{\star} = \mu_{n-2}{\ell}\ell' \quad \ell' \in \Pi_1 .
\end{equation}
Hence if two nodes $P,Q\in\Xset$ use the line ${\ell}$ then there exists
a third node $R\in\Xset$ using it and all the nodes $\Yset:=\Xset\setminus \{A,B,C,P,Q,R\}$ are located in a
maximal curve $\mu_{n-2}$ of degree $n-2:$
\begin{equation}\label{mu}\Yset\subset\mu_{n-2}.\end{equation}

Next, let us show that there is no
fourth node using ${\ell}$. We will prove this by the way of
contradiction. Assume that, except of the nodes $P,Q,R,$ there is a fourth node $S$ that uses ${\ell}.$ Of course we have that
$S\in\Yset.$

Then $P$ and $S$ are using $\ell$ therefore, as was proved above, there exists
a third node $T\in\Xset$ (which may coincide or not with $Q$ or $R$) using it and all the nodes of $\tilde\Yset:=\Xset\setminus \{A,B,C,P,S,T\}$ are located in a
maximal curve $\tilde\mu_{n-2}$ of degree $n-2.$ We have also that
\begin{equation}\label{tilde}
p_{S}^{\star} = \tilde\mu_{n-2}{\ell}\ell'' \quad \ell'' \in \Pi_1 .
\end{equation}

Now, notice that both $\mu_{n-2}$ and $\tilde\mu_{n-2}$ pass through all the nodes of the set
$\Zset:=\Xset\setminus \{A,B,C,P,Q,R,S,T\}$ with $\#\Zset\ge N-8.$

Now, according to the
Corollary~\ref{maincor}, with $k=n-2$, $N-8 = d(n, n-3)+2$ nodes determine
the curve of degree $n-2$ passing through them uniquely. Thus $\mu_{n-2}$ and $\tilde\mu_{n-2}$ coincide.

Therefore, in view of \eqref{mu} and \eqref{tilde}, $p_{S}^{\star}$ vanishes at all the nodes of $\Yset$,
which is a contradiction since $S\in\Yset.$
\end{proof}

%Bibliography

{\noindent H. Hakopian, S. Toroyan\\
Department of Informatics and Applied Mathematics\\
Yerevan State University\\
A. Manukyan St. 1\\
0025 Yerevan, Armenia \\}

\noindent E-mails - \texttt{hakop@ysu.am}, \quad
\texttt{sofitoroyan@gmail.com}
\end{document}